\documentclass[12pt, reqno]{amsart}
\usepackage{fullpage,url,amssymb}
\usepackage[backref]{hyperref}
\usepackage[alphabetic,backrefs,lite]{amsrefs}

\usepackage{color}

\newtheorem{lemma}{Lemma}[section]

\newtheorem{claim*}{Claim}

\newtheorem{thm}[lemma]{Theorem}

\newcommand{\Q}{{\mathbb Q}}

\newcommand{\Z}{{\mathbb Z}}

\newcommand{\pp}{{\mathfrak p}}

\newcommand{\OO}{{\mathcal O}}

\DeclareMathOperator{\inv}{inv}

\DeclareMathOperator{\Br}{Br}

\DeclareMathOperator{\Sym}{Sym}

\DeclareMathOperator{\ev}{ev}
\DeclareMathOperator{\Az}{Az}
\DeclareMathOperator{\PGL}{PGL}

\numberwithin{equation}{section}
\numberwithin{table}{section}

\newcommand{\defi}[1]{\textsf{#1}} 

\title[Failure of the Hasse principle]{Failure of the Hasse principle for Ch\^atelet surfaces in characteristic $2$}

\author[B. Viray]{Bianca Viray}

\address{Department of Mathematics, University of California, Berkeley, CA 94720-3840, USA}

\email{bviray@math.berkeley.edu}

\urladdr{http://math.berkeley.edu/~bviray}

\date{}

\thanks{This research was supported by the Mentored Research Award from UC Berkeley}

\begin{document}

\begin{abstract}
Given any global field $k$ of characteristic $2$, we construct a Ch\^atelet surface over $k$ that fails to satisfy the Hasse principle.  This failure is due to a Brauer-Manin obstruction. This construction extends a result of Poonen to characteristic $2$, thereby showing that the \'etale-Brauer obstruction is insufficient to explain all failures of the Hasse principle over a global field of any characteristic.
\end{abstract}



\maketitle

\section{Introduction}
Poonen recently showed that, for a global field $k$ of characteristic different from $2$, the \'etale-Brauer obstruction is insufficient to explain failures of the Hasse principle~\cite{poonen-insufficiency}.  This result relied on the existence of a Ch\^atelet surface over $k$ that violates the Hasse principle~\cite[Prop 5.1 and \S 11]{poonen-chatelet}.  Poonen's construction fails in characteristic $2$ due to the inseparability of $y^2-az^2$.  

Classically, Ch\^atelet surfaces have only been studied over fields of characteristic different from $2$.  In this paper, we define a Ch\^atelet surface over fields of characteristic $2$ and obtain a result analogous to~\cite[Prop 5.1]{poonen-chatelet}.
\begin{thm}\label{thm:HPcounterex}
	Let $k$ be any global field of characteristic $2$.  There exists a Ch\^atelet surface $X$ over $k$ that violates the Hasse principle.
\end{thm}


The only assumption on characteristic in~\cite{poonen-insufficiency} is in using~\cite[Prop 5.1]{poonen-chatelet} (all other arguments go through exactly as stated after replacing any polynomial of the form $by^2 +az^2$ by its Artin-Schreier analogue, $by^2+byz+az^2$).  Therefore, Theorem~\ref{thm:HPcounterex} extends the main result of~\cite{poonen-insufficiency} to global fields of characteristic $2$, thereby showing that the \'etale-Brauer obstruction is insufficient to explain all failures of the Hasse principle over a global field of any characteristic.

The proof of Theorem~\ref{thm:HPcounterex} is constructive.  The difficulty in the proof lies in finding suitable equations so that the Brauer set is easy to compute and empty.

\section{Background}%

\subsection{Brauer-Manin obstructions}%
The counterexamples to the Hasse principle referred to in Theorem~\ref{thm:HPcounterex} are all explained by the Brauer-Manin obstruction, which we recall here~\cite[Thm. 1]{manin-BMobs}.   Let $k$ be a global field and let $\mathbb{A}_k$ be the ad\`ele ring of $k$.  Recall that for a projective variety $X$, we have the equality $X(\mathbb A_k) = \prod _v X(k_v)$, where $v$ runs over all nontrivial places of $k$. The Brauer group of $X$, denoted $\Br X$, is the group of equivalence classes of Azumaya algebras on $X$.  Let $\inv_v$ denote the map $\Br \Q_v \to \Q/\Z$. Define
\[
	X(\mathbb A_k)^{\Br} := \left \{ (P_v)_v \in X(\mathbb A_k) \colon \sum_v \inv_v\left(\ev_{\mathcal A}(P_v)\right) = 0 \textrm{ for all } \mathcal A \in \Br X\right\}
\]
By class field theory we have
\[
	X(k) \subseteq X(\mathbb A_k)^{\Br} \subseteq X(\mathbb A_k).
\]
Thus, if $X(\mathbb A_k)^{\Br}=\emptyset$, then $X$ has no $k$-points.  We say there is a \defi{Brauer-Manin obstruction to the Hasse principle} if $X(\mathbb A_k) \neq \emptyset$ but $X(\mathbb A_k)^{\Br} = \emptyset$.  See~\cite[\S 5.2]{skorobogatov-torsors} for more details.

\subsection{Ch\^atelet surfaces in characteristic $2$}%
A conic bundle $X$ over $\mathbb P^1$ is the zero-locus of a nowhere-vanishing global section $s$ of $\Sym^2(\mathcal E)$ in $\mathbb P\mathcal E$, for some rank $3$ vector sheaf $\mathcal E$ on $\mathbb P^1$.  Consider the special case where $\mathcal E = \mathcal O \oplus \mathcal O\oplus \mathcal O(2)$ and $s = s_1 -s_2$ where $s_1$ is a global section of $\Sym^2(\mathcal O \oplus\mathcal O)$ and $s_2$ is a global section of $\mathcal O(2)^{\otimes 2} = \mathcal O(4)$.  Take $a\in k^{\times}$ and $P(x)$ a separable polynomial over $k$ of degree $3$ or $4$.  If $s_1 = y^2+yz+az^2$ and $s_2 = w^4P(x/w)$, then $X$ contains the affine variety defined by $y^2+yz+az^2 = P(x)$ as an open subset.  In this case we say $X$ is the Ch\^atelet surface defined by 
\[
	y^2+yz+az^2=P(x).
\]
By the same basic argument used in~\cite[Lemma 3.1]{poonen-chatelet}, we can show that $X$ is smooth.  See~\cite[\S 3 and \S 5]{poonen-chatelet} for the construction of a Ch\^atelet surface in the case where the characteristic is different from $2$.

\section{Proof of Theorem~\ref{thm:HPcounterex}}
Let $k$ denote a global field of characteristic $2$.  Let $\mathbb F$ denote its constant field and let $n$ denote the order of $\mathbb F^{\times}$.  Fix a prime $\pp$ of $k$ of odd degree and let $S = \{\pp\}$.  Let $\OO_{k,S}$ denote the ring of $S$-integers.  Let $\gamma \in \mathbb{F}$ be such that $T^2+T+\gamma$ is irreducible in $\mathbb{F}[T]$.  By the Chebotarev density theorem~\cite[Thm 13.4, p. 545]{neukirch-ant}, we can find elements $a,b \in \mathcal{O}_{k,S}$ that generate prime ideals of even and odd degree, respectively, such that $a \equiv \gamma \pmod{b^2\mathcal O_{k,S}}$.  These conditions imply that $v_\pp(a)$ is even and negative and that $v_\pp(b)$ is odd and negative.

Define
\begin{eqnarray*}
	f(x) & = & a^{-4n}bx^2 + x + ab^{-1},\\
	g(x) & = & a^{-8n}b^2x^2+ a^{-4n}bx + a^{1-4n} + \gamma.
\end{eqnarray*}
Note that $g(x) = a^{-4n}bf(x) + \gamma.$  Let $X$ be the Ch\^atelet surface given by
\begin{equation*}\tag{$*$}\label{eq:chatelet}
	y^2 + yz + \gamma z^2 = f(x)g(x).
\end{equation*}
In Lemma~\ref{lem:local-solvability} we show $X(\mathbb{A}_k) \neq \emptyset$, and in Lemma~\ref{lem:inv-map} we show $X(\mathbb A_k)^{\Br} = \emptyset$.  Together, these show that $X$ has a Brauer-Manin obstruction to the Hasse principle.

\begin{lemma}\label{lem:local-solvability}
	The Ch\^atelet surface $X$ has a $k_v$-point for every place $v$.
\end{lemma}
\begin{proof}
	Suppose that $v = v_a$.  Since $ a$ generates a prime of even degree, the left-hand side of~(\ref{eq:chatelet}) factors in $k_v[y,z]$.  Therefore, there is a solution over $k_v$.
	
	Now suppose that $v\neq v_a$.  Since $y^2+yz +az^2$ is a norm form for an unramified extension of $k_v$ for all $v$, in order to prove the existence of a $k_v$-point, it suffices to find an $x \in k_v$ such that the valuation of the right-hand side of~(\ref{eq:chatelet}) is even.
	
	Suppose further that $v \neq v_\pp, v_b$.  Choose $x$ such that $v(x) = -1.$  Then the right-hand side of~(\ref{eq:chatelet}) has valuation $-4$ so there exists a $k_v$-point.
	
	Suppose that $v = v_{\pp}$.  Let $\pi$ be a uniformizer for $v$ and take $x = \pi a^2/b.$  Then
	\begin{eqnarray*}
		f(x) & = & b^{-1}a^{4-4n}\pi^2 + a^2b^{-1}\pi + ab^{-1}.
	\end{eqnarray*}
	Since $a$ has negative even valuation and $n \ge 1$,  we have $v(f(x)) = v(a^2b^{-1}\pi)$ which is even.  Now let us consider
	\begin{eqnarray*}
		g(x) & = & a^{4-8n}\pi^2 + a^{2-4n}\pi + a^{1-4n}+\gamma.
	\end{eqnarray*}
	By the same conditions mentioned above, all terms except for $\gamma$ have positive valuation.  Therefore $v(g(x)) = 0$.
	
	Finally suppose that $v = v_{b}$.  Take $x = \frac{1}{b} +1$.  Then
	\[
		f(x) = \frac{1}{b}\left( a^{-4n}+a+1 + b + a^{-4n}b^2\right). 
	\]
	Note that by the conditions imposed on $a$, $\left( a^{-4n}+a+1 + b + a^{-4n}b^2\right) \equiv \gamma + b \pmod{b^2 \mathcal O_{k,S}}$.  Thus $v(f(x))=-1.$  Now consider
	\[
		g(x) = a^{-8n} + a^{-8n}b^2+ a^{-4n} + a^{-4n}b + a^{1-4n} + \gamma
	\]
	modulo $b^2\mathcal O_{k,S}$.  By the conditions imposed on $a$, we have 
\[
	g(x) \equiv 1+1+b+\gamma+\gamma \equiv b\pmod{b^2\mathcal O_{k,S}}.
\]
	Thus $v(g(x))=1$, so $v\left(f(x)g(x)\right)$ is even.
\end{proof}	

Let $L = k[T]/(T^2+T+\gamma)$ and let $\mathcal A$ denote the class of the cyclic algebra $\left(L/k, f(x)\right)_2$ in $\Br k(X)$.  Using the defining equation of the surface, we can show that $\left(L/k, g(x)\right)_2$ is also a representative for $\mathcal A$.  Since $g(x) + a^{-4n}bf(x)$ is a $v$-adic unit, $g(x)$ and $f(x)$ have no common zeroes.  Since $\mathcal A$ is the class of a cyclic algebra of order $2$, the algebra $\left(L/k, f(x)/x^2\right)_2$ is another representative for $\mathcal A$.  Note that for any point $p$ of $X$, there exists an open neighborhood $U$ containing $P$ such that either $f(x)$, $g(x)$, or $f(x)/x^2$ is a nowhere vanishing regular function on $U$.  Therefore, $\mathcal A$ is an element of $\Br X$.

To show that $X(\mathbb A_k)^{\Br} = \emptyset$, we use the continuity of the map $\ev_{\mathcal A}$.  While the result is well-known, it is difficult to find in the literature so we give a proof for reader's convenience.

\begin{lemma}\label{lem:ev-cont}
	Let $k_v$ be a local field and let $V$ be a smooth projective scheme over $k_v$.  For any $[\mathcal B] \in \Br_{\Az} V$,
	\[
		\ev_{\mathcal B} : V(k_v) \to \Br k_v
	\]
	is continuous for the discrete topology on $\Br k_v$.
\end{lemma}

\begin{proof}
	To prove continuity, it suffices to show that $\ev_{\mathcal B}^{-1}(\mathcal B')$ is open for any $\mathcal B'$ in the image of $\ev_{\mathcal B}$.  By replacing $[\mathcal B]$ with $[\mathcal B] - [\ev_{\mathcal B}(x)]$, we reduce to showing that $\ev_{\mathcal B}^{-1}(0)$ is open.  
	
	Fix a representative $\mathcal B$ of the element $[\mathcal B]\in \Br_{\Az} V$.  Let $n^2$ denote the rank of $\mathcal B$ and let $f_{\mathcal B}\colon Y_{\mathcal B} \to V$ be the $\PGL_n$-torsor associated to $\mathcal B$.  Then we observe that the set $\ev_{\mathcal B}^{-1}(0)$ is equal to $f_{\mathcal B}(Y_{\mathcal B}(k_v)) \subset V(k_v)$.  This set is open by the implicit function theorem.
\end{proof}

\begin{lemma}\label{lem:inv-map}
	Let $P_v \in X(k_v)$.  Then
	\[
	\inv_v(\ev_{\mathcal A}(P_v)) = 
		\begin{cases}
				1/2 & \text{if $v=v_b$},\\
				0 & \text{otherwise}.
		\end{cases}
	\]
	Therefore $X(\mathbb A_k)^{\Br} = \emptyset$.
\end{lemma}
\begin{proof}
	The surface $X$ contains an open affine subset that can be identified with
	\[
	V(y^2+yz+az^2 -P(x)) \subseteq \mathbb A^3.
	\]
	Let $X_0$ denote this open subset.  Since $\ev_{\mathcal A}$ is continuous by Lemma~\ref{lem:ev-cont} and $\inv_v$ is an isomorphism onto its image, it suffices to prove that $\inv_v$ takes the desired value on the $v$-adically dense subset $X_0(k_v)\subset X(k_v)$.  
	
	Since $L/k$ is an unramified extension for all places $v$, evaluating the invariant map reduces to computing the parity of the valuation of $f(x)$ or $g(x)$.  
	
	Suppose that $v\neq v_a,v_b, v_\pp$.  If $v(x_0) < 0$, then by the strong triangle inequality, $v(f(x_0))= v(x_0^2)$.  Now suppose that $v(x_0) \geq 0$.  Then both $f(x_0)$ and $g(x_0)$ are $v$-adic integers, but since $g(x) - a^{-4n}bf(x) = \gamma$ either $f(x_0)$ or $g(x_0)$ is a $v$-adic unit.  Thus, for all $P_v \in X_0(k_v)$, $\inv_v(\mathcal A(P_v))=0$.
	
	Suppose that $v = v_a$.  Since $a$ generates a prime of even degree, $T^2+T+\gamma$ splits in $k_a$.  Therefore, $(L/k, h)$ is trivial for any $h \in k_a(V)^\times$ and so $\inv_v(\mathcal A(P_v))=0$ for all $P_v \in X_0(k_v)$.
	
	Suppose that $v = v_\pp$.  We will use the representative $\left(L/k, g(x)\right)$ of $\mathcal A$.  If $v(x_0) < v(a^{4n}b^{-1})$ then the quadratic term of $g(x_0)$ has even valuation and dominates the other terms.  If $v(x_0) > v(a^{4n}b^{-1})$ then the constant term of $g(x_0)$ has even valuation and dominates the other terms.  Now assume that $x_0 = a^{4n}b^{-1}u$, where $u$ is $v$-adic unit.  Then we have
	\[
		g(x_0) = u^2 + u + \gamma + a^{1-4n}.
	\]
	Since $\gamma$ was chosen such that $T^2+T+\gamma$ is irreducible in $\mathbb{F}[T]$ and $\pp$ is a prime of odd degree, $T^2 + T + \gamma$ is irreducible in $\mathbb{F}_\pp[T]$.  Thus, for any $v$-adic unit $u$, $u^2+u+\gamma \not\equiv 0 \pmod \pp.$  Since $a \equiv 0 \mod \pp$, this shows $g(x_0)$ is a $v$-adic unit.  Hence $\inv_v(\mathcal A(P_v))=0$ for all $P_v \in X_0(k_v)$.
	
	Finally suppose that $v = v_b$.  We will use the representative $\left(L/k, f(x)\right)$ of $\mathcal A$.  If $v(x_0) < -1$ then the quadratic term has odd valuation and dominates the other terms in $f(x_0)$.  If $v(x_0) > -1$ then the constant term has odd valuation and dominates the other terms in $f(x_0)$.  Now assume $x_0 = b^{-1}u$ where $u$ is any $v$-adic unit.  Then we have
	\[
		f(x_0) = \frac{1}{b}\left ( a^{-4n}u^2 + u + a\right).
	\]
	It suffices to show that $a^{-4n}u^2 + u + a \not \equiv 0 \pmod{ b\mathcal O_{k,S}}.$  Since $a \equiv \gamma \pmod{b\mathcal O_{k,S}},$ we have 
\[
	a^{-4n}u^2 + u + a \equiv \overline{u}^2 + \overline{u} +\gamma.
\]
	Using the same argument as in the previous case, we see that $a^{-4n}u^2 + u + a \not \equiv 0 \pmod{b\mathcal O_{k,s}}$ and thus $v(g(x_0))=-1$.  Therefore $\inv_v(\mathcal A(P_v))=\frac{1}{2}$ for all $P_v \in X_0(k_v)$.
\end{proof}

\section*{Acknowledgements}
	I thank my advisor, Bjorn Poonen, for suggesting the problem and for many helpful conversations.  I thank Olivier Wittenberg for a sketch of the proof of Lemma~\ref{lem:ev-cont} and Daniel Erman for comments improving the exposition.

\end{document}